\documentclass{amsart}
\usepackage{chngcntr}
\usepackage{apptools}
\usepackage{color}
\usepackage{amsthm,amssymb,verbatim}
\usepackage{graphicx}
\usepackage{enumerate}
\usepackage{chngcntr}
\usepackage{apptools}
\usepackage{color}
\usepackage{amsthm,amssymb,verbatim}
\usepackage{graphicx}
\usepackage{enumerate}

\newcommand{\Ff}{\mathcal F}

 \newcommand{\Cc}{\mathcal C}

 \newcommand{\M}{\operatorname{M}}
 
 \newcommand{\Var}{\operatorname{var}}

\newcommand{\<}{\left<}
\renewcommand{\>}{\right>}

\newcommand{\Pp}{\mathcal{P}}

 \newcommand{\Bb}{\mathcal B}
 \newcommand{\RR}{\mathbf{R}}  

\newcommand{\ee}{\mathbf e}

\newcommand{\Hh}{\mathcal{H}}
\usepackage{amsthm}

\input epsf
\def\begfig {
\begin{figure}
\small }
\def\endfig {
\normalsize
\end{figure}
}

    \newtheorem{theorem}    {Theorem}   
    
    \newtheorem{corollary}  [theorem]     {Corollary}

    \newtheorem*{theorem*}{Theorem}
    \theoremstyle{definition}

    \theoremstyle{definition}
    
\title{Stationary polyhedral varifolds minimize area}
\author{Brian White}
\thanks{The author was partially supported by NSF grant DMS-1711293}
\address{Department of Mathematics\\ Stanford University\\ Stanford, CA 94305}
\email{bcwhite@stanford.edu}
\date{November 29, 2019.  Revised January 11, 2020.}
\usepackage{hyperref}
\usepackage{enumerate}
\usepackage[alphabetic, msc-links, backrefs]{amsrefs}
\begin{document}

\maketitle

\begin{abstract}
We prove that every stationary polyhedral varifold minimizes area in the
following senses:
 (1) its area cannot be decreased by
a one-to-one Lipschitz ambient deformation that coincides with the identity
outside of a compact set, and 
 (2) it is the varifold associated to a mass-minimizing flat chain with coefficients
in a certain metric abelian group.
\end{abstract}

NOTE: After this paper was written, I learned that
(1) above was proved by Choe~\cite{choe}
and that (2) was proved by Morgan~\cite{morgan}.
Hence this note should be viewed as an exposition of their results.
(Perhaps the result in~\S7 is new.)

\section{Introduction}

The tangent cone to any $2$-dimensional minimal variety (i.e., stationary varifold with a positive
   lower bound on density)
is polyhedral  (because the link is a stationary $1$-varifold in the unit sphere, 
and thus is a finite union of geodesic arcs~\cite{allard-almgren}.)
 If the variety is area-minimizing in some sense, then the cone is also area-minimizing.
Thus it is natural to ask: which $2$-dimensional stationary cones, or, more generally, which $m$-dimensional
stationary polyhedral cones, are area-minimizing in some sense?

In her celebrated 1976 paper~\cite{taylor-soap}
about soap films, Jean Taylor classified all two-dimensional, multiplicity-one
 polyhedral cones
in $\RR^3$ that minimize area in the following sense: the area (without counting multiplicity) cannot be 
decreased by a compactly supported Lipschitz deformation.  The Lipschitz deformation need not be one-to-one.
She showed that there are only three such cones: the plane, the union of three halflplanes meeting at equal angles
along their common edge, and the cone over a regular tetrahedron.
However, there are many other stationary polyhedral cones.  For example, in $\RR^3$, there
are seven other $2$-dimensional, multiplicity-one stationary cones that have, away from the vertex, only triple-junction-type singularities~\cite{taylor-soap}*{pp.~501, 502}.
Taylor's theorem leaves open the possibility that those other cones could be area-minimizing in
some other sense.   For example, one could ask
\begin{enumerate}
\item Which stationary polyhedral cones minimize area with respect to 
deformations by compactly supported Lipschitz homeomorphisms of the ambient space?
\item Which stationary polyhedral cones can be assigned orientations and multiplicities
in some coefficient group in such that a way that the cone is mass-minimizing as a flat chain
for that coefficent group?
\end{enumerate}
In this paper, we show that {\bf all} stationary polyhedral cones have both of those properties.
Indeed, the properties hold for every stationary polyhedral varifold, whether or not conical.

In particular, we prove

\begin{theorem}[Choe~\cite{choe}]\label{ONE}
Suppose that $V$ is an $m$-dimensional, rectifiable varifold in an open subset $U$ of $\RR^N$.
Suppose that $V$ is stationary, that $V$ is supported in a finite union of affine $m$-planes,
and that $V$ has finite mass.
Then 
\[
    \M(V) \le \M(f_\#V)
\]
for every Lipschitz homeomorphism $f: U\to U$ such that $\{x: \phi(x)\ne x\}$ has compact closure in $U$.
\end{theorem}

Here $\M(V)$ is the total mass of the varifold $V$.  (Thus $\M(V)$ is $\mu_V(\RR^N)$
 in the notation of~\cite{simon-book}
or $\|V\|(\RR^N)$ in the notation of~\cite{allard}.)

Theorem~\ref{ONE} is a consequence of the following:

\begin{theorem}[Morgan~\cite{morgan}]\label{TWO}
Suppose that $\Gamma$ is a closed subset of $\RR^N$ consisting of a finite union of $(m-1)$-dimensional
polyhedra.
Suppose that $V$ is an $m$-dimensional, compactly supported, rectifiable varifold in $\RR^N$ such that
\begin{enumerate}
\item $V$ is supported in a finite union of affine $m$-planes, and
\item $V$ is stationary in $\RR^N\setminus \Gamma$.
\end{enumerate}
Then there is a metric abelian group $G$ and a polyhedral $m$-chain $A$ with coefficients in $G$ such that
\begin{enumerate}
\item $V$ is the varifold associated to $A$.
\item $\partial A$ is supported in $\Gamma$.
\item $A$ is mass-mininimizing: 
if $A'$ is any other flat $m$-chain (with coefficients in $G$) such that $\partial A'=\partial A$, then $\M(A)\le \M(A')$.
\end{enumerate}
We can choose the coefficient group $G$ to be a certain finite-dimensional Euclidean space 
  (namely $\Lambda_n\RR^N$) with its associated norm.
Alternatively, we can choose $G$ to be a discrete group with the property:
\begin{equation}\label{property}
\text{ $\{g\in G: |g|\le \lambda\}$ is finite for every $\lambda<\infty$}.  \tag{*}
\end{equation}
If the varifold is an integral varifold, then we can also require that $|g|$ is an integer for every $g\in G$.
\end{theorem}

(Theorem~\ref{ONE} and Theorem~\ref{TWO} with $G=\Lambda_m\RR^N$ are proved in \S\ref{main-section}.
  Theorem~\ref{changing-theorem} in \S\ref{changing-section} 
shows how one can then construct from $\Lambda_n\RR^N$ a suitable coefficient group that has Property~\eqref{property}.)

For either choice of coefficient group, the resulting space of flat chains satisfies the compactness
theorem: given any sequence of flat chains supported in a compact set with mass and boundary mass bounded above,
there is a subsequence that converges in the flat topology.
However, if $G$ is a normed vector space over $\RR$,
 then there will be finite-mass flat chains with coefficients in $G$
that fail to be rectifiable.
On the other hand, if the coefficient group has Property~\eqref{property}, then every finite mass flat chain is rectifiable.
(These assertions about rectifiablity and lack of rectifiability follow from~\cite{white-rectifiability}*{Theorem~7.1}.)
Thus coefficient groups with Property~\eqref{property}
 provide a nice setting for the Plateau problem: mass minimizing surfaces
(for any boundary) exist, and they must be rectifiable.
This paper shows that every stationary polyhedral cone arises as the varifold associated to a mass-minimizing cone
for such a coefficient group.

This paper is meant to be largely self-contained.
See~\cite{white-fluids} for a brief overview of flat chain with coefficients in a metric abelian group,
or
Fleming's original paper~\cite{fleming} for a thorough treatment.

I would like to thank Giovanni Alberti for stimulating discussions on the topic of this paper.

\section{Abstract Calibrations}

Let $\Cc$ be a collection (e.g., of surfaces), 
let $\M:\Cc\to\RR$ be a real-valued function (e.g., ``mass" or ``weighted
area"), and let $\sim$ be an equivalence relation on $\Cc$ (e.g, the relation ``is homologous to".)
An {\bf abstract calibration} on $(\Cc,M,\sim)$ is a function 
\[
  \Phi: \Cc\to \RR
\]
such that 
\begin{enumerate}
\item $\Phi$ is constant on each equivalence class, and
\item $\Phi(S)\le \M(S)$ for every $S\in\Cc$.
\end{enumerate}
If $\Phi(S)=\M(S)$, we say that $S$ is {\bf calibrated} by $\Phi$.

Abstract calibrations are of interest because of the following trivial theorem:

\begin{theorem}
Suppose that $\Phi:\Cc\to\RR$ is a calibration on $(\Cc,\M,\sim)$
and that $S\in \Cc$ is calibrated by $\Phi$. 
Then $S$ minimizes $\M$ in its equivalence class:
\[
   \M(S)\le \M(S') \quad\text{for all $S' \sim S$}.
\]
Furthermore, if $S'\sim S$, then $S'$ minimizes $\M$ in its equivalence class if and only if $S'$ is also calibrated by
  $\Phi$.
\end{theorem}

Now suppose that $\Cc$ is an abelian group and that $\mathcal{B}$ is a subgroup.
(Think of $\Cc$ as a collection of $m$-chains and $\Bb$
as the chains in $\Cc$ that are boundaries of $(m+1)$-chains.)
Then we have the equivalence relation on $\Cc$ given by 
\[
   S \sim S' \quad\text{if and only if}\quad S-S' \in \Bb. 
\]
Now suppose that
\[
  \Phi: \Cc \to\RR
\]
is a homomorphism that vanishes on $\mathcal{B}$.   
Then $\Phi$ is constant on equivalence classes for the equivalence relation~$\sim$.
Thus if $\Phi(\cdot)\le \M(\cdot)$, then $\Phi$ is an {\bf abstract calibration} for $(\Cc,\M, \sim)$.
In that case, we will call $\Phi$ a (generalized) {\bf calibration}.

\section{Polyhedral Chains}

Fix an ambient Euclidean space $\RR^N$ and
 an abelian group $(G,+)$ with a norm $|\cdot|$, i.e., with a function
\[
   |\cdot|: G\to\RR
\]
such that the distance function $\operatorname{d}(g,g'):= |g-g'|$ makes $G$ into a metric space.

A {\bf formal polyhedral $m$-chain} (in $\RR^N$, with coefficients in $G$) is a formal sum
\[
  \sum_{i=1}^k g_i \sigma_i
\]
where $g_i\in G$ and $\sigma_i$ is an oriented $m$-dimensional polyhedron in $E$.
Let $\Pp_m^\textnormal{formal}=\Pp_m^\textnormal{formal}(G)$
 be the abelian group of formal polyhedral $m$-chains in $\RR^N$
with coefficients in $G$.

Consider the equivalence relation on formal polyhedral chains generated by 
\[
    g\sigma \equiv g\sigma' + g \sigma''
\]
if $\sigma'$ and $\sigma''$ are obtained from $\sigma$ by subdivision, and by
\[
   g\sigma \equiv - g\tilde \sigma
\]
is $\tilde \sigma$ is obtained from $\sigma$ by reversing the orientation.

A {\bf polyhedral $m$-chain} is an equivalence class in $\Pp_m^\textnormal{formal}$ under
this equivalence relation.   We let $\Pp_m=\Pp_m(G)$ be the abelian group of polyhedral $m$-chains
with coefficients in $G$.

If $\sum_{i=1}^k g_i \sigma_i$ is a formal polyhedral $m$-chain, we let 
\[
  \sum_{i=1}^k g_i [\sigma_i]
\]
denote the corresponding equivalence class, i.e., the corresponding polyhedral chain.
Every polyhedral chain has a representation $\sum_ig_i[\sigma_i]$ in which
the $\sigma_i$ are non-overlapping.   For such a representation,
the mass (or weighted area) of the chain is given by
\[
 \M\left(\sum g_i[\sigma_i]\right) := \sum |g_i|\,\Hh^m(\sigma_i).
\]

\section{A Canonical Calibration}\label{canonical-section}

In this section, we fix nonnegative integers $m$ and $N$ with $m<N$.
Let $G$ be the additive group $\Lambda_m\RR^N$ of $m$-vectors in $\RR^N$.
We give $\Lambda_m\RR^N$ the Euclidean norm.  That is, we give $\Lambda_m\RR^N$ the norm associated
to the Euclidean structure
for which the $m$-vectors $\ee_{i_1}\wedge \dots \wedge \ee_{i_m} (i_1<i_2<\dots < i_m)$
form an orthonormal basis of $\Lambda_m\RR^N$.

(For the study of non-rectifiable chains, it would probably be better to use the mass norm~\cite{federer-fleming}*{page~461} on $\Lambda_n\RR^N$.
However, in this paper we are primarily interested in rectifiable chains, and 
whether one uses the Euclidean norm or the mass norm does not affect the masses
of rectifiable chains.)

If $\sigma$ is an oriented $m$-dimensional polyhedron in $\RR^N$, we let $\eta(\sigma)$ be the
corresponding simple unit $m$-vector.

Now define a map
\[
   \Phi: \Pp_m^\textnormal{formal}(G) \to \RR
\]
by
\[
   \Phi\left(\sum_i g_i \sigma_i \right) = \sum_i (g_i\cdot \eta(\sigma_i)) \Hh^m{\sigma_i}.
\]
Note that $\Phi$ is an additive homomorphism.

If $\tilde \sigma$ is obtained from $\sigma$ by reversing the orientation, then $\eta(\tilde \sigma)=-\eta(\sigma)$,
so
\[
    (-g)\cdot\eta(\tilde \sigma) = g \cdot \sigma,
\]
and therefore 
\[
   \Phi(-g\tilde \sigma) = \Phi(g\sigma).
\]
Similarly, 
\[
   \Phi(g\sigma) = \Phi(g\sigma') + \Phi(g\sigma'')
\]
if $\sigma'$ and $\sigma''$ are obtained by subdividing $\sigma$.

Hence $\Phi$ induces a well-defined homomorphism
\begin{align*}
&\Phi: \Pp_m(G) \to \RR, \\
&\Phi\left(\sum_i g_i[\sigma_i]\right) = \sum (g_i\cdot \eta(\sigma_i)) \Hh^m(\sigma_i).
\end{align*}

\begin{theorem}\label{calibration-criterion}
Suppose that $G=\Lambda_m\RR^N$ and that
\[
  A = \sum g_i [\sigma_i]
\]
is a polyhedral chain in $\Pp_m(G)$, 
where the $\sigma_i$ are non-overlapping.
Then
\[
   \Phi(A)\le \M(A),
\]
with equality if and only if each $g_i$ is a nonnegative scalar multiple of $\eta(\sigma_i)$.
\end{theorem}

\begin{proof}
By Cauchy-Schwartz, $g_i\cdot\eta(\sigma_i)\le |g_i|$, with equality if and only if $g_i$
is a nonnegative scalar multiple of $\eta(\sigma_i)$.  Hence
\begin{align*}
\Phi(A) 
&= \sum (g_i\cdot\eta(\sigma_i)) \Hh^m(\sigma_i)  \\
&\le \sum |g_i| \Hh^m(\sigma_i) \\
&= \M(A).
\end{align*}
with equality if and only if each $g_i$ is a nonnegative scalar multiple of $\eta(\sigma_i)$.
\end{proof}

\begin{theorem}\label{stokes-theorem}
Suppose that $G=\Lambda_m\RR^N$ and that $A\in \Pp_m(G)$ is the boundary of an $(m+1)$-chain. 
 Then $\Phi(A)=0$.
\end{theorem}

\begin{proof}
Let $g\in G$, let $\tau$ be an oriented $(m+1)$-dimensional polyhedron in $E$,
and let $\sigma_1,\dots,\sigma_k$ be the $m$-dimensional faces of $\tau$ with the induced orientations.
It suffices to show that
\begin{equation}\label{zero}
   \Phi\left(\sum_i g [\sigma_i]\right) = 0.
\end{equation}
Since any $m$-vector is a sum of simple $m$-vectors, it suffices to prove~\eqref{zero}
 when $g$ is a simple $m$-vector.
Let $V$ be the oriented $m$-plane associated to $g$.   Let $\omega$ be the volume form on $V$,
$\Pi: \RR^N\to V$ be orthogonal projection, and let $\Omega = \Pi^\#\omega$.
Note that
\[
     \Phi(g[\sigma_i]) = |g| \int_{\sigma_i}\Omega.
\]
Thus
\begin{align*}
\Phi\left(\sum_i g[\sigma_i]\right) 
&= |g|\sum \int_{\sigma_i}\Omega  \\
&= |g| \int_{\partial \tau}\Omega  \\
&= 0
\end{align*}
by Stokes Theorem (since $\Omega$ is constant and therefore $d\Omega=0$.)
\end{proof}

Recall that the {\bf flat norm} of $A\in \Pp_m(G)$ is given by
\[
  \Ff(A) = \inf_{Q\in \Pp_{m+1}(G)} (\M(A + \partial Q) + \M(Q)),
\]
which is trivially $\le \M(A)$.

\begin{theorem}\label{Phi-bound-theorem} If $G=\Lambda_m\RR^N$ and $A\in \Pp_m(G)$, then
\[
  \Phi(A) \le \Ff(A) \le \M(A),
\]
where $\Ff(A)$ is the flat norm of $A$.
\end{theorem}

\begin{proof}
Let $Q$ be a polyhedral $(m+1)$-chain.  By Theorem~\ref{stokes-theorem},
\begin{align*}
\Phi(A) 
&= \Phi(A + \partial Q) \\
&\le \M(A+\partial Q) \\
&\le \M(A+\partial Q) + \M(Q).
\end{align*}
Hence 
\[
 \Phi(A) \le \inf_Q (\M(A+\partial Q) + \M(Q)) = \Ff(A).
\]
\end{proof}

\begin{corollary}
The map $\Phi$ extends continuously to an additive homomorphism $\Phi: \Ff_m(G)\to \RR$ such that
\[
  \Phi(A)\le \Ff(A)\le \M(A)
\]
for every flat $m$-chain $A$.
\end{corollary}

The corollary follows immediately from the theorem since $\Ff_m(G)$ is the metric space
completion of $\Pp_m(G)$ with respect the flat norm.

\section{Polyhedral Varifolds}

If $\sigma$ is an $m$-dimensional polyhedron in $\RR^N$, we let $\Var(\sigma)$
 be the associated $m$-dimensional, multiplicity-$1$
rectifiable varifold, i.e., the rectifiable varifold whose associated Radon measure is $\Hh^m\llcorner \sigma$.
An $m$-dimensional {\bf polyhedral varifold} in $\RR^N$ is a varifold of the form
\[
   \sum_{i=1}^k c_i \Var(\sigma_i),
\]
where each $\sigma_i$ is a polyhedron and each $c_i\ge 0$.  By subdividing, we can assume that 
if $i\ne j$, then $\sigma_i\cap\sigma_j$ is either empty or is a common face of $\sigma_i$ and $\sigma_j$
of dimension $<m$.

As in \S\ref{canonical-section}, we let $G=\Lambda_m\RR^N$.  
If $\sigma$ is an oriented $m$-dimensional polyhedron in $\RR^N$, we let $\eta(\sigma)$
be the simple unit $m$-vector associated with the orientation.   Let $\left<\sigma\right>$ be the polyhedral
$m$-chain in $\Pp_m(G)$ given by
\[
   \left<\sigma\right>  =  \eta(\sigma) [\sigma].
\]
Note that if $\tilde \sigma$ is obtained from $\sigma$ by reversing the orientation, then 
$\<\tilde \sigma\>=\<\sigma\>$.   Thus given an unoriented $m$-dimensional polyhedron $\sigma$, 
we have a well-defined polyhedral $m$-chain $\<\sigma\>$ in $\Pp_m(G)$.
 
If
\[
  V = \sum_{i=1}^k c_i \Var(\sigma_i)
\]
is an $m$-dimensional polyhedral varifold, we let
\[
  \<V\> = \sum_{i=1}^k c_i \<\sigma_i\>.
\]
If we give the $\sigma_i$ orientations, then
\[
   \<V\> = \sum_{i=1}^k c_i \eta(\sigma_i)[\sigma_i],
\]
from which it follows (by Theorem~\ref{calibration-criterion}) that $V$ is calibrated by $\Phi$.
Note that $\M(\<V\>)=\M(V)$.
Thus $\<\,\cdot\,\>$ is a mass-preserving homomorphism from the additive semigroup of $m$-dimensional polyhedral 
varifolds to the additive group $\Pp_m(G)$ of $m$-dimensional polyhedral chains.

\begin{theorem}\label{boundary-support}
Let $V$ be an $m$-dimensional polyhedral varifold in $\RR^N$, and let
$\Gamma$ be the union of some of the $(m-1)$-dimensional faces of polyhedra in $V$.
Then $V$ is stationary in $\RR^N\setminus \Gamma$ if and only if $\partial \<V\>$ is supported in $\Gamma$.
\end{theorem}

\begin{proof}
Since the result is essentially local, it suffices to prove it for 
\[
   V = \sum_i c_i\Var(\sigma_i)
\]
where the $\sigma_i$ are polyhedra with a common $(m-1)$-dimensional face $\tau$
and where $\Gamma$ is the union of the faces $\ne \tau$ of the various $\sigma_i$.

Give $\tau$ an orientation, and then give each $\sigma_i$ the orientation that induces the chosen
orientation on $\tau$.   
Let $\eta(\tau)$ be the simple unit $(m-1)$-vector associated to the orientation of $\tau$
and let $\eta(\sigma_i)$ be the simple unit $m$-vector associated with the orientation of $\sigma_i$.
Then
\begin{equation}\label{nu-equation}
    \eta(\sigma_i) = \eta(\tau)\wedge \nu_i, 
\end{equation}
where $\nu_i$ is the unit normal to $\tau$ that lies in the $m$-plane containing $\sigma_i$ and that points
out from $\sigma_i$.

Note that
\begin{equation}\label{stationary-equation}
\begin{gathered}
\text{$V$ is stationary in $\Gamma^c$} \\
\iff 
\sum c_i \nu_i = 0 \\
\iff
\sum c_i\eta(\sigma_i)=0.
\end{gathered}
\end{equation}
On the other hand, 
\begin{align*}
( \partial \<\Var(\sigma_i)\> ) \llcorner \Gamma^c
&=
(\partial \eta(\sigma_i) [\sigma_i])\llcorner \Gamma^c  \\
&=
\eta(\sigma_i) [\tau],
\end{align*}
so
\begin{equation}\label{boundary-equation}
 ( \partial \<V\> ) \llcorner \Gamma^c
 =
 \left( \sum c_i \eta(\sigma_i) \right) [\tau].
\end{equation}
The theorem follows immediately from~\eqref{stationary-equation} and~\eqref{boundary-equation}.
 
\end{proof}

\section{The Main Theorem}\label{main-section}

\begin{theorem}\label{main-theorem}
Let $G=\Lambda_m\RR^N$.
If $V$ is an $m$-dimensional polyhedral varifold, then $\<V\>$ is a mass-minimizing polyhedral chain in 
$\Pp_m(G)$.   If $\Gamma$ is closed set (such as the union of some of the faces of polyhedra in $V$)
and if $V$ is stationary in $\Gamma^c$, then $\partial \<V\>$ is supported in $\Gamma$ and
\[
  \M(V) \le \M(\phi_\#V)
\]
for any lipschitz homeomorphism $\phi: \RR^N\to\RR^N$ such that $\phi(x)=x$ for all $x\in \Gamma$.
\end{theorem}

\begin{proof}
By construction, $\<V\>$ is calibrated by $\Phi$.
Hence $\<V\>$ is homologically mass-minimizing.  Since the ambient space $\RR^N$ is homologically
trivial, that means that $\<V\>$ minimizes mass, i.e., that $\M(\<V\>)\le \M(A)$ for any flat
chain $A$ in $\Ff_m(G)$ with $\partial A=\partial \<V\>$.

Now suppose that $V$ is stationary in $\Gamma^c$.  Then $\partial \<V\>$ is supported in $\Gamma$
(by Theorem~\ref{boundary-support}),
so 
\[
   \M(\<V\>) \le \M( f_\# \<V\>)
\]
for any Lipschitz map $f:\RR^N\to\RR^N$ such that $f(x)=x$ for all $x\in \Gamma$.
Now $\M(\<V\>)=\M(V)$, and, if $f$ is one-to-one, then $\M(f_\#\<V\>) = \M(f_\#V)$.
Thus 
\[
   \M(V) \le \M(f_\#V).
\]
\end{proof}

\section{Changing the Coefficient Group}\label{changing-section}

\begin{theorem}\label{changing-theorem}
Let $G$ be an abelian group with norm $|\cdot|_G$.
Let $S=\{g_1,\dots,g_k\}$ be a finite subset of $G$, and let $H$ be the subgroup of $G$ generated by $S$.
Define a norm $|\cdot|_H$ on $H$ by
\[
  |g|_H = \min \left\{ \sum_{i=1}^k |n_i|\,|g_i|:  g = \sum_{i=1}^k n_i g_i \right\},
\]
where the $n_i$ are integers.
Suppose that $A$ is a rectifiable flat chain in $\Ff_m(G)$ and that all of its multiplicities
are in $S$.  Then $A$ may also be regarded as a rectifiable flat chain in $\Ff_m(H)$
(with the same mass).   Furthermore, if $A$ is mass-minimizing in $\Ff_m(G)$, then
it is also mass-minimizing in $\Ff_m(H)$.
\end{theorem}

Note that $\{g\in H: |g|_H\le \lambda\}$ is finite if $\lambda<\infty$.
Note also that $|\cdot|_H$ is the largest norm on $H$ such that $|g_i|_H=|g_i|_G$
for each of the generators $g_1,\dots,g_k$.

\begin{proof}
Let $\M_G$ and $\Ff_G$ denote mass and flat norm on chains (polyhedral or flat) with
coefficients in the group $G$ with the norm $|\cdot|_G$.   Let $\M_H$ and $\Ff_H$ denote mass and flat norm
on chains with coefficients in the group $H$ with the norm $|\cdot|_H$.
Since $\M_H\ge \M_G$ on $\Pp_m(H)$, it follows that 
  $\Ff_H\ge \Ff_G$ on $\Pp_m(H)$.  
Consequently,  
the inclusion
\[
\Pp_m(H) \subset   \Pp_m(G)
\]
extends to an inclusion
\[
  \Ff_m(H) \subset \Ff_m(G)
\]
such that $\M_G(T)\le \M_H(T)$ and $\Ff_G(T)\le \Ff_H(T)$ for $T\in \Ff_m(H)$.

Now suppose that $A$ is an $\M_G$-minimizing rectifiable chain in $\Ff_m(G)$ and that the multiplicities
of $A$ lie in $S$.   Then $A$ is also a rectifiable chain in $\Ff_m(H)$, and $\M_H(A)=\M_G(A)$.
If $A'$ is a chain in $\Ff_m(H)$ with $\partial A'=\partial A$, then
\[
   \M_G(A)\le\M_G(A')
\]
since $A$ is $\M_G$-minimizing. 
Since $\M_G(A)=\M_H(A)$ and $\M_G(A')\le \M_H(A')$, it follows
that
\[
   \M_H(A) \le \M_H(A').
\]
Hence $A$ is $\M_H$-minimizing in $\Ff_m(H)$.
\end{proof}

\nocite{pedrosa-ritore}
\nocite{hoffman-wei}
\newcommand{\hide}[1]{}

\end{document}